\documentclass[12pt, notitlepage]{amsart}
\usepackage{latexsym, amsfonts, amsmath, amssymb, amsthm, cite, tabls, paralist}

\newtheorem{theorem}{Theorem}[section]
\newtheorem{lemma}[theorem]{Lemma}
\newtheorem{proposition}[theorem]{Proposition}

\newtheorem{corollary}[theorem]{Corollary}
\newtheorem{conjecture}[theorem]{Conjecture}

\theoremstyle{remark}
\newtheorem{remark}[theorem]{Remark}

\theoremstyle{definition}

\theoremstyle{remark}
\newtheorem{example}[theorem]{Example}

\numberwithin{equation}{section}

\newcommand{\Z}{\mathbb{Z}}
\newcommand{\Q}{\mathbb{Q}}
\newcommand{\R}{\mathbb{R}}
\newcommand{\prob}{\mathbf{P}}
\newcommand{\F}{\mathbb{F}}

\newcommand{\E}{\mathbb{E}}

\newcommand{\pp}{\mathfrak{p}}
\newcommand{\PP}{\mathfrak{P}}
\newcommand{\QQ}{\mathfrak{Q}}
\newcommand{\Gal}{\text{Gal}}

\usepackage{graphicx}
\usepackage{mathrsfs}

\usepackage{varioref}

\newcommand{\un}{\underline}

\begin{document}

\title[Sharpness of the larger sieve]{Polynomial values modulo primes on average and sharpness of the larger sieve}
\author{Xuancheng Shao}
\address{Department of Mathematics \\ Stanford University \\
450 Serra Mall, Bldg. 380\\ Stanford, CA 94305-2125}
\email{xshao@math.stanford.edu}

\maketitle

\begin{abstract}
This paper is motivated by the following question in sieve theory. Given a subset $X\subset [N]$ and $\alpha\in (0,1/2)$. Suppose that $|X\pmod p|\leq (\alpha+o(1))p$ for every prime $p$. How large can $X$ be? On the one hand, we have the bound $|X|\ll_{\alpha}N^{\alpha}$ from Gallagher's larger sieve. On the other hand, we prove, assuming the truth of an inverse sieve conjecture, that the bound above can be improved (for example, to $|X|\ll_{\alpha}N^{O(\alpha^{2014})}$ for small $\alpha$). The result follows from studying the average size of $|X\pmod p|$ as $p$ varies, when $X=f(\Z)\cap [N]$ is the value set of a polynomial $f(x)\in\Z[x]$.
\end{abstract}

\section{Introduction}

For a positive integer $N$, denote by $[N]$ the set $\{1,2,\cdots,N\}$. The letter $p$ is always used to denote a prime. To primary goal of this paper is to study upper bounds for the sizes of subsets $X\subset [N]$ occupying a small fraction of residue classes modulo many primes $p$. Gallagher's larger sieve \cite{Gal71} provides such an upper bound.

\begin{theorem}[Larger sieve]\label{thm:gallagher}
Let $X\subset [N]$ be a subset and $\mathcal{P}$ be a set of primes. We have
\[ |A|\leq\frac{\sum_{p\in\mathcal{P}}\log p}{\sum_{p\in\mathcal{P}}|X\pmod p|^{-1}\log p-\log N}, \]
whenever the denominator is positive.
\end{theorem}

See \cite{CE04} for some variants of it and references therein for applications. We are particularly interested in the situation when $|X\pmod p|\leq \alpha p$ for some fixed $\alpha\in (0,1)$, and whether the bound provided by the larger sieve is best possible. 

\begin{corollary}[Larger sieve, special case]\label{cor:gallagher}
Let $X\subset [N]$ be a subset and $\alpha\in (0,1/2]$. If $|X\pmod p|\leq (\alpha+o(1)) p$ for every prime $p$, then $|X|\ll N^{\alpha+o(1)}$.
\end{corollary}

This is easily deduced from Theorem \ref{thm:gallagher} by taking $\mathcal{P}$ to be the set of primes up to $N^{\alpha+o(1)}$. When $\alpha>1/2$, the statement still holds, but is beaten by the bound $|X|\ll_{\alpha}N^{1/2}$ following from the large sieve \cite{Mon78}. When $\alpha\leq 1/2$, is the bound $|X|\ll N^{\alpha+o(1)}$ sharp? If $X$ is the set of perfect squares up to $N$, then $|X|\sim N^{1/2}$ and $X$ occupies $(p+1)/2$ residue classes (the quadratic residues) modulo any odd prime $p$. The question of whether this is the only type of sharp example is usually referred to as the inverse sieve conjecture, informally stated as follows.

\begin{conjecture}[Inverse sieve conjecture, rough form]
Let $X\subset [N]$ be a subset. If $|X\pmod p|\leq 0.9p$ for every prime $p$, then either one of the following two statements holds:
\begin{enumerate}
\item the cardinality of $X$ is extremely small;
\item the set $X$ possesses algebraic structure.
\end{enumerate} 
\end{conjecture}

See Conjecture \ref{conj:inverse} below for one precise formulation of it. See also \cite{CL07,HV09,Wal12,GH14} for more discussions and evidences towards it.

Now assume that $\alpha<1/2$ is fixed. Motivated by the inverse sieve conjecture, we consider the sizes of $X\pmod p$ when $X$ is the value set of a polynomial. For a polynomial $f(x)\in \Z[x]$ of degree $d\geq 1$, denote by $f_p\in\F_p[x]$ the reduction $f\pmod p$. Let $\alpha_p(f)=p^{-1}|f_p(\F_p)|$, the relative size of the value set of $f\pmod p$. Define $\alpha(f)$ to be the average of $\alpha_p(f)$ as $p$ varies:
\[ \alpha(f)=\lim_{Q\rightarrow\infty}\frac{1}{\pi(Q)}\sum_{p\leq Q}\alpha_p(f). \]
Note the trivial lower bounds $\alpha_p(f)\geq d^{-1}$ and $\alpha(f)\geq d^{-1}$. 

\begin{theorem}[Polynomial values modulo primes on average]\label{thm:poly}
Let $f\in \Z[x]$ be a polynomial of degree $d\geq 1$. Then
\begin{equation}\label{eq:alpha-rec} 
\lim_{Q\rightarrow\infty} \frac{1}{\pi(Q)}\sum_{p\leq Q}\alpha_p(f)^{-1}\leq \tau(d), 
\end{equation}
where $\tau(d)$ is the number of positive divisors of $d$. Consequently, we have $\alpha(f)\geq\tau(d)^{-1}$.
\end{theorem}

Note that for $d\geq 3$, we always have $\tau(d)<d$. Hence it is reasonable to conjecture that Corollary \ref{cor:gallagher} is not sharp whenever $\alpha$ is smaller than (and bounded away from) $1/2$. See the last section in \cite{Sha14} for a preliminary discussion on the simplest case $d=3$.

\begin{theorem}[Inverse sieve conjecture implies improved larger sieve]\label{thm:improved-larger}
Assume the truth of Conjecture \ref{conj:inverse}. Let $X\subset [N]$ be a subset and $\alpha\in (0,1)$. Let $\epsilon\in (0,1)$ be a parameter. If $|X\pmod p|\leq (\alpha+o(1))p$ for every prime $p$, then $|X|\ll_{\alpha,\epsilon} N^{1/d}$ where $d$ is the smallest positive integer with $\tau(d)\geq (1-\epsilon)\alpha^{-1}$.
\end{theorem}

Since $\tau(d)\leq d^{C/\log\log d}$ for some constant $C>0$, the conclusion above implies that $|X|\ll_{\alpha}N^{\alpha^{c\log\log\alpha^{-1}}}$ for some constant $c>0$, a huge improvement upon Corollary \ref{cor:gallagher} for small $\alpha$ (assuming the truth of the inverse sieve conjecture).

In the remainder of this introduction we discuss further about the quantities $\alpha_p(f)$ and $\alpha(f)$. Note that\eqref{eq:alpha-rec} becomes an equality when $f(x)=x^d$. Indeed, in this case we have $\alpha_p(f)\sim (p-1,d)^{-1}$, and thus the average of $\alpha_p(f)^{-1}$ is equal to
\[ \frac{1}{\phi(d)}\sum_{a\in (\Z/d\Z)^{\times}} (a-1,d)=\tau(d). \]
Note, however, that in this case the average of $\alpha_p(f)$ is equal to
\[ \alpha(f)=\frac{1}{\phi(d)}\sum_{a\in (\Z/d\Z)^{\times}} (a-1,d)^{-1},   \]
which can be evaluated to $\phi(d)/d$ when $d$ is squarefree (and at least $(\phi(d)/d)^2$ for any $d$). Since $\phi(d)/d\gg (\log\log d)^{-1}$, the following construction provides polynomials $f$ with smaller $\alpha(f)$.

\begin{theorem}[Polynomials with small value sets modulo primes]\label{thm:poly-good}
Define a sequence of polynomials $\{f_n\}$ by
\[ f_1(x)=x^2,\ \ f_{n+1}(x)=(f_n(x)+1)^2. \]
Then $\alpha_p(f_n)=a_n$ provided that $p>2f_{n-1}(0)+2$ when $n>1$, where the sequence $\{a_n\}$ is defined by
\[ a_1=\frac{1}{2},\ \ a_{n+1}=a_n-\frac{a_n^2}{2}. \]
Moreover, we have $a_n\leq 2n^{-1}$ for each $n$.
\end{theorem}

Since $\deg f_n=2^n$, we have $\alpha(f_n)\ll (\log (\deg f_n))^{-1}$. We do not know whether this is the best example or whether the bound for $\alpha(f)$ in Theorem \ref{thm:poly} is sharp. See Section \ref{sec:rem-alpha} below for more discussions on this.

The investigation of $\alpha_p(f)$ for a fixed prime $p$ has a long history (see \cite{BS59,Coh70}), and explicit formulae for $\alpha_p(f)$ are known in terms of the proportion of fixed point free elements in a certain Galois group (see Lemma \ref{lem:cohen} below and the remark afterwards). Not surprisingly, the quantity $\alpha(f)$ can also be evaluated in terms of a certain Galois group, and this is recorded in Proposition \ref{prop:alpha-average} below. Due to a lack of understanding of the relevant Galois groups, our lower bound for $\alpha(f)$ is instead obtained by studying the number of solutions to $f(x)\equiv f(y)\pmod p$ on average as $p$ varies (see Section \ref{sec:statements}), and it is for this reason that the average of $\alpha_p(f)^{-1}$ naturally shows up.

A related line of work is on classifying those polynomials $f\in\F_p[x]$ for which $\alpha_p(f)$ is close to the lower bound $d^{-1}$ (for a fixed $p$). In particular, results in \cite{GM88} imply that $\alpha_p(f)\geq 2d^{-1}+o(1)$ whenever $p\not\equiv \pm 1\pmod d$. 

The rest of this paper is organized as follows. In Section \ref{sec:statements} we state a general and quantitative version of Theorem \ref{thm:poly} for polynomials over arbitrary number fields, and outline the proof strategy, with the details given in Section \ref{sec:average}. In Section \ref{sec:inverse-sieve} we state a precise form of the inverse sieve conjecture and deduce Theorem \ref{thm:improved-larger}. In Section \ref{sec:good-poly}, Theorem \ref{thm:poly-good} is proved by computing relevant Galois groups. Finally in Section \ref{sec:remarks}, we make some further remarks concerning the larger sieve as well as the quantity $\alpha(f)$.

\smallskip

\textbf{Acknowlegements.} Thanks to Brian Conrad for help with proofs, to Kannan Soundararajan for helpful discussions, and to Akshay Venkatesh for asking a question that leads to this paper.

\section{Statement of results and proof strategy}\label{sec:statements}

Let $K$ be a number field and $\mathcal{O}_K$ be its ring of integers. Denote by $\Delta_K$ the (absolute) discriminant of $K$. We will use the letter $\pp$ to denote a prime ideal in $\mathcal{O}_K$, $\kappa_{\pp}$ to denote the residue field $\mathcal{O}_K/\pp$, and $N(\pp)=|\kappa_{\pp}|$ to denote the norm of $\pp$. For a polynomial $f(x)\in \mathcal{O}_K[x]$ of degree $d$, define $f_{\pp}$ and $\alpha_{\pp}(f)$ as in the introduction. To recall, $f_{\pp}$ is the reduction of $f$ modulo $\pp$ and $\alpha_{\pp}(f)=N(\pp)^{-1}|f_{\pp}(\kappa_{\pp})|$.

To make our result quantitative, we also need a notion that measures the sizes of the coefficients of $f$. For this purpose, we define the (absolute logarithmic) height $h(f)$ of $f\in\mathcal{O}_K[x]$ to be the sum
\begin{equation}\label{eq:height}
h(f)=\sum_{v} \max_a \log |a|_v,
\end{equation}
where the sum is over all places $v$ of $K$ and the maximum is taken over all coefficients $a$ of $f$. The quantity $|a|_v$ is normalized such that it does not depend on the field $K$. For example, when $f\in\Z[x]$ is primitive, the height $h(f)$ is the logarithm of the absolute value of the largest coefficient of $f$. See \cite{HS00} for basic properties of the height function.

Instead of studying $\alpha_{\pp}(f)$ directly, we find it easier to study the related quantity $m_{\pp}(f)$, defined to be
\[ m_{\pp}(f)=N(\pp)^{-1}\cdot \#\{(x,x')\in\kappa_{\pp}\times\kappa_{\pp}:f_{\pp}(x)=f_{\pp}(x')\}. \]
By an application of Cauchy-Schwarz inequality, we have $\alpha_{\pp}(f)\geq m_{\pp}(f)^{-1}$.  Therefore Theorem \ref{thm:poly} is a consequence of the following:

\begin{theorem}[Average number of solutions modulo primes]\label{thm:m-average-special}
Let $K$ be a number field and $f(x)\in \mathcal{O}_K[x]$ be a polynomial of degree $d\geq 1$. Let $s(f)$ be the number of irreducible factors of $f(x)-f(y)$ in $K[x,y]$. Then for any $Q\geq 2$, we have
\[ \sum_{N(\pp)\leq Q}m_{\pp}(f)= s(f)\sum_{N(\pp)\leq Q}1+O(Q\exp(-c\sqrt{\log Q})+h(f)),      \]
for sufficiently small $c=c(K,d)>0$.
\end{theorem}

To see that this implies Theorem \ref{thm:poly}, note that $s(f)\leq\tau(d)$ (when $K=\Q$) since the homogeneous part of degree $d$ of $f(x)-f(y)$ is $a(x^d-y^d)$ for some nonzero $a$ and it factors into $\tau(d)$ irreducible factors (which are cyclotomic polynomials).

We will in fact prove the following more general result, of which Theorem \ref{thm:m-average-special} is a special case. For a multivariable polynomial $g(\un{X})\in \mathcal{O}_K[\un{X}]$ in $n$ variables of total degree $d$, define $g_{\pp}$ and $m_{\pp}(g)$ similarly as above. More precisely, $g_{\pp}$ is the reduction of $g$ modulo $\pp$ and 
\[ m_{\pp}(g)=N(\pp)^{-(n-1)}\cdot \#\{\un{X}\in\kappa_{\pp}^n: g_{\pp}(\un{X})=0\}. \]
Define the (absolute logarithmic) height $h(g)$ of $g$ as in \eqref{eq:height}. Two polynomials $g_1,g_2$ are said to be equivalent if they are scalar multiples of each other.

\begin{theorem}[Average number of solutions modulo primes; general form]\label{thm:m-average}
Let $K$ be a number field and $g(\un{X})\in \mathcal{O}_K[\un{X}]$ be a polynomial in $n$ variables of total degree $d\geq 1$. Let $s(g)$ be the number of non-equivalent irreducible factors of $g$ in $K[\un{X}]$. Let $L$ be a Galois extension of $K$ such that $g$ factors into absolutely irreducible factors in $L[\un{X}]$. Let $C=C(K,n,d)>0$ be sufficiently large. If $Q\geq \exp(C(\log \Delta_L)^2)$, then
\begin{equation}\label{eq:m-average} 
\sum_{N(\pp^m)\leq Q}m_{\pp}(g)\log N(\pp)=s(g)Q-t(g)\frac{Q^{\beta_0}}{\beta_0}+O(Q\exp(-c\sqrt{\log Q})+h(g)+\log\Delta_L),
\end{equation}
for sufficiently small $c=c(K,n,d)>0$, where $t(g)\in [0,s(g)]$, and the second term appears only if the Dedekind zeta function $\zeta_L$ has a Siegel zero $\beta_0\in (1/2,1)$. Consequently, for $Q\geq \exp(C(\log\Delta_L)^2)$ we have
\begin{equation}\label{eq:m-average2} 
\sum_{N(\pp)\leq Q}m_{\pp}(g)\leq s(g)\sum_{N(\pp)\leq Q}1+O(Q\exp(-c\sqrt{\log Q})+h(g)+\log\Delta_L).     
\end{equation}
\end{theorem}

The bounds for the error terms stem from a quantitative version of Chebotarev density theorem in \cite{LO77}. Assuming the truth of the Generalized Riemann Hypothesis (GRH) for $\zeta_L$, we can get a much better error term $O(Q^{1/2}(\log \Delta_L+[L:Q]\log Q))$, and of course without the Siegel zero term. The unconditional error term, however, is already enough for our application. 


\begin{proof}[Proof of Theorem \ref{thm:m-average-special} assuming Theorem \ref{thm:m-average}]
We show that $f(x)-f(y)$ factors into absolutely irreducible factors over $L=K(\mu_d)$, where $\mu_d$ is the group of $d$th roots of unity. Indeed, since the homogeneous part of degree $d$ of $f(x)-f(y)$ is $a(x^d-y^d)$ for some nonzero $a\in K$, it factors over $L$ into linear factors. Thus there is a factorization
\[ f(x)-f(y)=\prod_{i=1}^r g_i(x,y) \]
of $f(x)-f(y)$ into absolutely irreducible factors $g_1,g_2,\cdots,g_r$, such that the top degree part of each $g_i$ is defined over $L$. We claim that each $g_i$ is defined over $L$ as well. Suppose not. Without loss of generality, assume that some coefficient of $g_1$ does not lie in $L$. Let $\tau\in\Gal(\overline{Q}/L)$ be an automorphism that moves this coefficient. Let $\tau(g_1)$ be the polynomial obtained by applying $\tau$ to every coefficient of $g_1$. Then $\tau(g_1)$ is also a factor of $f(x)-f(y)$, and thus $\tau(g_1)$ is equivalent to $g_i$ for some $1\leq i\leq r$. By our choice of $\tau$, $\tau(g_1)$ must be equivalent to $g_i$ for some $i>1$, and thus $g_1$ and $g_i$ have equivalent top degree part. This contradicts the fact that $x^d-y^d$ has no repeated factors.

Now that the potential Siegel zero $\beta_0$ of $\zeta_L$ depends only on $K$ and $d$, the Siegel zero term in \eqref{eq:m-average} can be absorbed into the error term, and the conclusion follows easily from partial summation.
\end{proof}

\begin{remark}
In the argument above we used the fact that polynomials of the form $f(x)-f(y)\in K[x,y]$ factors into absolutely irreducible factors in $L[x,y]$ with $L=K(\mu_d)$. For a general polynomial $g(\un{X})\in K[\un{X}]$ of height $h(g)$, it can be shown that one can take $L$ with $[L:\Q]\leq C$ and $\Delta_L\leq C\exp(Ch(g))$ for some constant $C=C(K,n,d)>0$. Thus the $\log\Delta_L$ factor in the error term can be removed. We will, however, not need this relation between the size of $L$ and the height $h(g)$. 
\end{remark}

\begin{remark}
The arguments in proving Theorem \ref{thm:m-average} can be generalized to study the average behavior of $|V(\F_p)|$ as $p$ varies, for any algebraic variety $V$ defined over $\Z$. More precisely, let $m=\dim V$. Then the average of $p^{-m}|V(\F_p)|$ as $p$ varies is equal to the number of irreducible components of $V$.
\end{remark}

To finish this section, we sketch the proof of Theorem \ref{thm:m-average}. By Lang-Weil, $m_{\pp}(g)$ is essentially the number of absolutely irreducible factors of $g_{\pp}$. Factor $g$ into absolutely irreducible factors in $L[\un{X}]$, and consider the natural action of the Galois group $G=\Gal(L/K)$ on these factors. For almost all primes $\PP\subset\mathcal{O}_L$, these absolutely irreducible factors remain absolutely irreducible modulo $\PP$, and thus $m_{\pp}(g)$ is essentially the number of these factors which are defined over $\kappa_\pp$. This is equal to the number of fixed points of the Frobenius element associated with $\PP$. By Chebotarev density theorem, these Frobenius elements are equidistributed in $G$ as $\PP$ varies. Hence the average of $m_{\pp}(g)$ is equal to the average number of fixed points of the $G$-action. By Burnside's lemma, this is equal to the number of $G$-orbits, which is exactly the number of irreducible factors $s(g)$ of $g$. In carrying out this procedure some additional efforts are needed to keep track of the explicit dependence on the height of $g$.

\section{Proof of Theorem \ref{thm:m-average}}\label{sec:average}

In this section we prove Theorem \ref{thm:m-average}. The implied constants appearing in this section are always allowed to depend on $K,n,d$.

Factor $(g)$ into principle prime ideals in $L[\un{X}]$:
\[ (g)=(g_1)^{e_1}(g_2)^{e_2}\cdots (g_r)^{e_r}, \]
where $g_i\in L[\un{X}]$ is absolutely irreducible, and $g_i,g_j$ are not equivalent when $i\neq j$. Let $G$ be the Galois group $\Gal(L/K)$. For any $1\leq i\leq r$ and any $\xi\in G$, let $\xi(g_i)$ be the polynomial obtained by applying $\xi$ to all coefficients of $g_i$. Since $\xi(g_i)$ is also a factor of $g$, $\xi(g_i)$ is equivalent to $g_j$ for some $1\leq j\leq r$. Hence $\xi$ acts on $\{(g_1),\cdots,(g_r)\}$ by sending $(g_i)$ to $(\xi(g_i))$. In this way we obtain a $G$-action on $\{(g_1),\cdots,(g_r)\}$.

\begin{lemma}[Galois descent]\label{lem:galois-descent}
Let $E$ be any field and $F$ be a Galois extension of $E$. Let $h\in F[\un{X}]$ be a polynomial. The following two statements are equivalent:
\begin{enumerate}
\item the ideal $(h)\subset F[\un{X}]$ is fixed by every element of $G=\Gal(F/E)$;
\item the ideal $(h)$ is defined over $E$. In other words, there exists a scalar $\alpha\in F^{\times}$ such that $\alpha h\in E[\un{X}]$.
\end{enumerate}
\end{lemma}

\begin{proof}
This is a standard result in the theory of Galois descent. For completeness, we give a proof here. Clearly (2) implies (1). Now assume that (1) holds, so that for each $\xi\in G$, we have $\xi(h)=c_{\xi}h$ for some $c_{\xi}\in F^{\times}$. The scalars $\{c_{\xi}:\xi\in G\}$ form a $1$-cocycle $G\rightarrow F^{\times}$, and thus by Hilbert's theorem 90, we have $c_{\xi}=\alpha/\xi(\alpha)$ for some $\alpha\in F^{\times}$. Now that $\xi(h)=\alpha h/\xi(\alpha)$, we conclude that $\xi(\alpha h)=\alpha h$ for each $\xi\in G$. Thus $\alpha h\in E[\un{X}]$, as desired.
\end{proof}

\begin{lemma}\label{lem:orbit}
Let the notations be as above. The number of orbits of the $G$-action on $\{(g_1),(g_2),\cdots,(g_r)\}$ is equal to $s(g)$.
\end{lemma}

\begin{proof}
Let $\mathcal{H}=\{h_1,h_2,\cdots,h_s\}$ be the set of non-equivalent irreducible factors of $g$ (well defined up to scalars in $K$), where $s=s(g)$. We construct a bijection between the set of orbits and $\mathcal{H}$. 

Let $\mathcal{O}\subset \{(g_1),(g_2),\cdots,(g_r)\}$ be a $G$-orbit, and let $h$ be the product of those $g_i$ with $(g_i)\in\mathcal{O}$. We claim that $(h)$ is defined over $K$, and moreover $(h)$ is a prime ideal in $K[\un{X}]$ (hence $(h)=(h_j)$ for some $1\leq j\leq s$). In fact, since any $\xi\in G$ permutes the factors in $\mathcal{O}$, the ideal $(h)$ is fixed by $\xi$. By Lemma \ref{lem:galois-descent}, the ideal $(h)$ is defined over $K$. Now let $h'\in K[\un{X}]$ be a factor of $h$ (with positive degree), and let $\mathcal{O}'\subset\mathcal{O}$ be the set of those $(g_i)\in\mathcal{O}$ dividing $h'$. For any $(g_i)\in\mathcal{O}'$ and any $\xi\in G$, $\xi(g_i)$ is also a factor of $h'$ and thus $(\xi(g_i))\in\mathcal{O}'$. This shows that $G$ preserves $\mathcal{O}'$, and thus $\mathcal{O}'=\mathcal{O}$ and $(h')=(h)$. Hence $(h)$ is a prime ideal.

Conversely, let $h_j\in\mathcal{H}$ be an irreducible factor of $g$, and let $\mathcal{O}$ be the set of those $(g_i)$ dividing $h_j$. We claim that $\mathcal{O}$ is a $G$-orbit, and moreover the product of those ideals in $\mathcal{O}$ is equal to $(h_j)$. In fact, for any $\xi\in G$ and $(g_i)\in \mathcal{O}$, the polynomial $\xi(g_i)$ is also a factor of $h_j$. Hence $G$ preserves $\mathcal{O}$. If $\mathcal{O}'\subset\mathcal{O}$ is a $G$-orbit, the argument above shows that the product of the ideals in $\mathcal{O}'$ is defined over $K$. Hence $\mathcal{O}'=\mathcal{O}$ by the irreducibility of $h_j$. Finally, the argument above also shows that the product of the ideals in $\mathcal{O}$ is defined over $K$, and is thus equal to $(h_j)$.
\end{proof}

The following lemma shows that the heights of the factors $g_i$ are controlled by the height of $g$. Note that the height $h(g_i)$ depends only on the ideal $(g_i)$ since two equivalent polynomials have the same height.

\begin{lemma}[Gelfond's inequality]\label{lem:gelfond}
Let the notations be as above. Then $h(g_i)\leq h(g)+C$ for some constant $C=C(K,n,d)>0$.
\end{lemma}

\begin{proof}
See Proposition B.7.3 in \cite{HS00}.
\end{proof}

Let $\pp$ be a prime in $\mathcal{O}_K$ and $\PP$ be a prime in $\mathcal{O}_L$ lying above $\pp$. For each $1\leq i\leq r$, let $(g_i)\pmod\PP$ be the ideal in $\kappa_{\PP}[\un{X}]$ obtained by reduction modulo $\PP$. The following lemma will be used to ensure that $(g_i)\pmod\PP$ remains absolutely irreducible for all but finitely many $\PP$.

\begin{lemma}[Noether]\label{lem:noether}
Let $n,d$ be positive integers. There exist polynomials $\ell_1,\cdots,\ell_m$ with integral coefficients depending only on $n$ and $d$ in variables $A_{i_1\cdots i_n}$ ($i_1+\cdots+i_n\leq d$), such that the following statement holds. For any algebraically closed field $\overline{F}$, a polynomial $f\in\overline{F}[\un{X}]$ in $n$ variables of total degree at most $d$ with
\[ f(x_1,\cdots,x_n)=\sum_{i_1+\cdots+i_n\leq d}a_{i_1\cdots i_n}x_1^{i_1}\cdots x_n^{i_n} \]
is reducible over $\overline{F}$ or has total degree less than $d$ if and only if $\ell_j((a_{i_1\cdots i_n}))=0$ for each $1\leq j\leq m$.
\end{lemma}

\begin{proof}
See Theorem 2A in \cite{Sch76}. 
\end{proof}

\begin{lemma}\label{lem:abs-irred}
Let the notations be as above. There exists a positive integer $E\leq C\exp(Ch(g))$ for some $C=C(K,n,d)>0$, such that $(g_i)\pmod\PP$ is absolutely irreducible for each $1\leq i\leq r$ whenever $\PP\nmid E$.
\end{lemma}

\begin{proof}
It suffices to prove the statement for each individual $i$. Let $\ell_1,\cdots,\ell_m$ be the polynomials in Lemma \ref{lem:noether} corresponding to the degree of $g_i$. After normalizing we may assume that some coefficient of $g_i$ is equal to $1$. Thus $h(a)\leq h(g_i)$ for every coefficient $a$ of $g_i$, where $h(a)$ for $a\in L^{\times}$ is defined by
\[ h(a)=\sum_v \max(\log |a|_v,0). \]
Since $g_i$ is absolutely irreducible, $\ell_j$ does not vanish at the coefficient vector of $g_i$ for some $1\leq j\leq m$; call this non-vanishing value $A\in L\setminus\{0\}$. Since all coefficients of $g_i$ have heights bounded by $h(g_i)$, we have $h(A)=O(h(g_i)+1)=O(h(g)+1)$. Therefore, there exists a positive integer $E\leq C\exp(Ch(g))$ such that $A\pmod \PP$ is nonzero whenever $\PP\nmid E$. For these $\PP$, the absolute irreducibility of $g_i\pmod\PP$ follows from another application of Lemma \ref{lem:noether}.
\end{proof}

\begin{remark}
Brian Conrad pointed out that (the qualitative version of) this is a special case of a general result in algebraic geometry: if $R$ is a domain with fraction field $F$ and $S$ is a domain finitely generated over $R$ such that the $F$-algebra $S_F=F\otimes_R S$ is absolutely irreducible over $F$, then there is a non-empty open subset $U\subset\text{Spec}(R)$ such that the fiber algebra $S_u=k(u)\otimes_R S$ over $k(u)$, the residue field at $u$, is absolutely irreducible. 
\end{remark}

Let $E$ be the positive integer from Lemma \ref{lem:abs-irred}. After enlarging $E$ if necessary (but still with $E\leq C\exp(Ch(g))$), we may assume that $g_1\pmod\PP,\cdots,g_r\pmod\PP$ are all non-equivalent to each other whenever $\PP\nmid E$.

Let $\pp\nmid E$ be a prime in $\mathcal{O}_K$ and $\PP$ be a prime in $\mathcal{O}_L$ lying above $\pp$. The decomposition group $G_{\PP}=\Gal(\kappa_{\PP}/\kappa_{\pp})$ acts on the factors $\{g_1\pmod\PP,\cdots,g_r\pmod\PP\}$ such that $\xi(g_i\pmod\PP)$ is equivalent to $g_j\pmod\PP$ for any $\xi\in G_{\PP}$. Via the natural inclusion $G_{\PP}\hookrightarrow G$, this action is compatible with the $G$-action on $\{g_1,\cdots,g_r\}$.

For any conjugacy class $[\xi]\subset G$, let $s([\xi])$ be the number of fixed points of any element in $[\xi]$. 

\begin{lemma}\label{lem:ms}
Let the notations be as above. If $\pp\nmid E$ and $\pp$ is unramified in $L$, then $m_{\pp}(g)=s([\sigma_{\pp}])+O(N(\pp)^{-1/2})$, where $[\sigma_{\pp}]$ is the Frobenius conjugacy class associated to $\pp$.
\end{lemma}

\begin{proof}
Let $h\in\{g_1,\cdots,g_r\}$. Note that $\sigma_{\PP}$ fixes $(h)$ if and only if $\sigma_{\PP}$ fixes $(h_{\PP})$, and this happens if and only if $(h_{\PP})$ is defined over $\kappa_{\pp}$ by Lemma \ref{lem:galois-descent}. Hence $s([\sigma_{\pp}])$ is exactly the number of non-equivalent absolutely irreducible factors of $g_{\pp}$ in $\kappa_{\pp}[\un{X}]$, and the conclusion follows from Lang-Weil.
\end{proof}

We are now ready to evaluate the quantity
\[ M_f(Q)=\sum_{N(\pp^m)\leq Q}m_{\pp}(g)\log N(\pp).  \]
By Lemma \ref{lem:ms}, we have
\[ M_f(Q)=\sum_{\substack{N(\pp^m)\leq Q\\ \pp\text{ unramified in }L}}s([\sigma_{\pp}]^m)\log N(\pp)+O(Q^{1/2}\log Q+\log E+\log \Delta_L).  \] 
Since $E\leq C\exp(Ch(g))$, we have $\log E=O(h(g)+1)$. Hence
\[ M_f(Q)=\sum_C s(C)\psi_C(Q)+O(Q^{1/2}\log Q+h(g)+\log\Delta_L), \]
where the sum is over all conjugacy classes $C$ in $G$, and
\[ \psi_C(Q)=\sum_{\substack{N(\pp^m)\leq Q\\ \pp\text{ unramified in }L\\ [\sigma_{\pp}]^m=C}}\log N(\pp). \]
By (a quantitative version of) the Chebotarev density theorem \cite{LO77}, for $Q\geq\exp(C(\log \Delta_L)^2)$ we have
\[ \psi_C(Q)=\frac{|C|}{|G|}Q-\frac{|C|}{|G|}\chi_0(C)\frac{Q^{\beta_0}}{\beta_0}+O(Q\exp(-c(\log Q)^{1/2})),  \]
where the second term occurs only if the Dedekind zeta function $\zeta_L$ has a Siegel zero $\beta_0$, and $\chi_0$ is the real character of a one-dimensional representation of $G$ for which the associated $L$-function has $\beta_0$ as a zero. It follows that
\[ M_f(Q)=Q\cdot \frac{1}{|G|}\sum_{\xi\in G}s(\xi)-\frac{Q^{\beta_0}}{\beta_0}\cdot\frac{1}{|G|}\sum_{\xi\in G}s(\xi)\chi_0(\xi)+O(Q\exp(-c(\log Q)^{1/2})+h(g)+\log\Delta_L). \]
By Burnside's lemma and Lemma \ref{lem:orbit}, we have
\[ \frac{1}{|G|}\sum_{\xi\in G}s(\xi)=s(g). \]
The equality \eqref{eq:m-average} follows by setting
\[ t(g)=\frac{1}{|G|}\sum_{\xi\in G}s(\xi)\chi_0(\xi). \]
By a change of summation, we can write
\[ t(g)=\frac{1}{|G|}\sum_{i=1}^r\sum_{\xi\in G_i}\chi_0(\xi), \]
where $G_i\subset G$ is the subgroup of elements fixing $(g_i)$. Since $\chi_0$ is a one-dimensional real character, the inner sum is either $0$ or $|G_i|$. Hence $t(g)\in [0,s(g)]$, as claimed. Finally, the inequality \eqref{eq:m-average2} follows easily from \eqref{eq:m-average} by dropping the Siegel zero term and partial summation.

\section{Inverse sieve conjecture implies improved larger sieve}\label{sec:inverse-sieve}

In this section we state a precise version of the inverse sieve conjecture and then prove Theorem \ref{thm:improved-larger}. The implied constants here are always allowed to depend on $\alpha,\epsilon$.

\begin{conjecture}[Inverse sieve conjecture]\label{conj:inverse}
Let $X\subset [N]$ be a subset and let $\epsilon>0$ be real. Assume that for each parameter $Q\geq N^{\epsilon}$, we have
\[ \sum_{p\leq Q}\frac{|X\pmod p|}{p}\leq (1-\epsilon)\pi(Q). \]
Then at least one of the following two situations happens:
\begin{enumerate}
\item(very small size) $|X|\ll_{\epsilon} N^{\epsilon}$;
\item (algebraic structure) there exists a polynomial $f(x)\in\Q[x]$ of degree $d\in [2,C]$ and height at most $N^C$ such that $|X\cap f([N])|\geq C^{-1}|X|$, where $C=C(\epsilon)$ is a constant.
\end{enumerate}
\end{conjecture}

Here, we say that a polynomial $f(x)\in\Q[x]$ has height at most $H$ if $f(x)=A^{-1}f^*(x)$ for some positive integer $A\leq H$ and $f^*\in\Z[x]$ with all coefficients bounded by $H$ in absolute value. This is slightly different from the notion of height used in the statement of Theorem \ref{thm:m-average-special}, in that $h(f)$ is invariant under scalar multiplication but the notion here is not. Note that if a polynomial $f(x)\in\Q[x]$ has height at most $H$, then $h(f)\ll \log H$.

\begin{remark}
We make a few remarks concerning why some quantitative aspects of this conjecture are reasonable.
\begin{itemize}
\item The condition on $X$ essentially says that $X$ misses a positive proportion of residue classes modulo primes $p$ on average, as soon as $p$ exceeds a small positive power of $N$. With this assumption we know from the large sieve that $|X|\ll N^{1/2}$ and from the larger sieve that $|X|\ll N^{\alpha+O(\epsilon)}$ if the upper bound $(1-\epsilon)\pi(Q)$ is replaced by $\alpha\pi(Q)$. Without the knowledge about $X\pmod p$ for $p\leq N^{\epsilon}$, one can essentially add to $X$ any $N^{\epsilon}$ extra elements without violating the assumption, but one should still expect to see algebraic structure apart from these extra elements.

\item The conclusion $|X\cap f([N])|\geq C^{-1}|X|$ is equivalent to the seemingly weaker one $|X\cap f(\Q)|\geq C^{-1}|X|$, after a suitable modification of the polynomial $f$ which does not increase its height too much. To see that the interval $[N]$ can be replaced by $\Z$, note that the set $J=\{n\in\Z:1\leq f(n)\leq N\}$ is the union of at most $d$ intervals and has size at most $dN$. Since $X\cap f(\Z)=X\cap f(J)$, there is an interval $I\subset J$ with $|X\cap f(I)|\geq d^{-1}|X\cap f(\Z)|$,  and we may assume that $I\subset [N]$ after a translation. To see that $f(\Z)$ can be replaced by $f(\Q)$, note that if $f(x)\in\Z$ for some $x\in\Q$ then the denominator of $x$ must divide some positive integer $B$ depending on the coefficient of $f$. Then $f(\Q)\cap\Z\subset f^*(\Z)\cap\Z$, where $f^*$ is defined by $f^*(x)=f(x/B)$.

\item The conclusion that $f(\Z)$ captures a positive proportion of $X$ cannot be replaced by the stronger one that $f(\Z)$ captures almost all of $X$. Indeed, it is possible for $X$ to be the union of $f(\Z)$ for several distinct polynomials $f$.
\end{itemize}
\end{remark}

If $|X\pmod p|\leq \alpha p$ for small $\alpha$, repeated applications of Conjecture \ref{conj:inverse} allows us to strengthen it by requiring the degree $d$ to be fairly large.

\begin{proposition}[Inverse sieve conjecture in the larger sieve regime]\label{prop:inverse-large-deg}
Assume the truth of Conjecture \ref{conj:inverse}. Let $X\subset [N]$ be a subset. Let $\alpha\in (0,1)$ and $\epsilon\in (0,\alpha)$ be real. Assume that $|X\pmod p|\leq (\alpha+o(1)) p$ for each prime $p$. Then at least one of the following two situations happens:
\begin{enumerate}
\item (very small size) $|X|\ll_{\epsilon} N^{\epsilon}$;
\item (algebraic structure) there exists a polynomial $f(x)\in\Q[x]$ of degree $d\in [2,C]$ and height at most $N^C$ such that $|X\cap f(\Z)|\geq C^{-1}|X|$, where $C=C(\epsilon)$ is a constant. Moreover, we may ensure that $\tau(d)\geq (1-\epsilon)\alpha^{-1}$. 
\end{enumerate}
\end{proposition}

\begin{proof}
Suppose that $|X|\gg N^{\epsilon}$. We will apply Conjecture \ref{conj:inverse} iteratively to construct a sequence of polynomials $f_1,f_2,\cdots,f_k$ and a sequence of sets $X_0=X,X_1,X_2,\cdots,X_k$ with $k=O(1)$ such that the following conditions hold:

\begin{enumerate}
\item $\deg f_i=d_i\in [2,C]$, and $\tau(d_1d_2\cdots d_k)\geq (1-\epsilon)\alpha^{-1}$;
\item the height of $f_i$ is at most $N^{O(1)}$ for each $1\leq i\leq k$;
\item $X_i\subset [N]$ and $|X_i|\gg |X_{i-1}|$ for each $1\leq i\leq k$;
\item $f_i(X_i)\subset X_{i-1}$ for each $1\leq i\leq k$.
\end{enumerate}

Suppose first that these objects are constructed. Let $f=f_1\circ f_2\circ\cdots\circ f_k$. By property (1), the degree $d$ of $f$ is $O(1)$ and satisfies $\tau(d)\geq (1-\epsilon)\alpha^{-1}$. By property (2), the height of $f$ is $N^{O(1)}$.  By property (3), we have $|X_k|\gg |X|$. By property (4), we have $f(X_k)\subset X\cap f([N])$. Hence 
\[ |X\cap f([N])|\geq |f(X_k)|\gg |X_k|\gg |X|, \]
as desired.

It thus remains to construct $f_1,\cdots,f_k$ and $X_1,\cdots,X_k$. Suppose that they are already chosen up to $f_{i-1}$ and $X_{i-1}$ for some $i\geq 1$ satisfying the required properties. We will construct $f_i$ and $X_i$ from those. Let $F=f_1\circ\cdots\circ f_{i-1}$ if $i>1$ and let $F$ be the identity map if $i=1$. Let $D$ be the degree of $F$. We may assume that $\tau(D)<(1-\epsilon)\alpha^{-1}$, since we may stop the iteration otherwise. By property (4), we have $F(X_{i-1})\subset X$. 

Let $F=A^{-1}F^*$ with $A\leq N^C$ a positive integer and $F^*\in\Z[x]$ a polynomial whose coefficients are all bounded by $N^C$. Let $p\nmid A$ be a prime. For each $r\in\Z/p\Z$, let $\nu_p(r)$ be the number of $x\in\Z/p\Z$ with $F(x)\equiv r\pmod p$. Then
\begin{align*} 
|X_{i-1}\pmod p| &\leq \sum_{r\in F(X_{i-1})\pmod p}\nu_p(r)\leq |X\pmod p|^{1/2}\left(\sum_r\nu_p(r)^2\right)^{1/2} \\
&\leq (\alpha+o(1))^{1/2} m_p(F^*)^{1/2}p,
\end{align*}
by Cauchy-Schwarz, the assumption that $|X\pmod p|\leq (\alpha+o(1)) p$, and the definition of $m_p(F^*)$. For any $Q\geq N^{\epsilon}$, we then have
\begin{align*} 
\sum_{p\leq Q}\frac{|X_{i-1}\pmod p|}{p} &\leq (\alpha+o(1))^{1/2}\sum_{p\leq Q}m_p(F^*)^{1/2}+O(\log A) \\
&\leq (\alpha+o(1))^{1/2}\pi(Q)^{1/2}\left(\sum_{p\leq Q}m_p(F^*)\right)^{1/2}+O(\log N).
\end{align*}
Now apply Theorem \ref{thm:m-average-special} to obtain
\[ \sum_{p\leq Q}\frac{|X_{i-1}\pmod p|}{p}\leq [(\alpha+o(1))\tau(D)]^{1/2}\pi(Q)+O(Q\exp(-c(\log Q)^{1/2})+Q^{1/2}\log N). \]
Since $\tau(D)<(1-\epsilon)\alpha^{-1}$, the first term above is at most $(1-\epsilon/2)\pi(Q)$, and thus $X_{i-1}$ satisfies the hypotheses in Conjecture \ref{conj:inverse}. Since $|X_{i-1}|\gg N^{\epsilon}$, we must be in the algebraic case. Let $f_i\in \Q[x]$ be a polynomial of degree $d_i\in [2,C]$ and height at most $N^C$ such that $|X_{i-1}\cap f_i([N])|\gg |X_{i-1}|$, and let $X_i\subset [N]$ be chosen such that $f_i(X_i)\subset X_{i-1}$ and $|X_i|\gg |X_{i-1}|$. This completes the inductive construction. Finally, since the quantity $\tau(d_1d_2\cdots d_i)$ strictly increases with $i$, the process terminates after $O(1)$ iterations. 
\end{proof}

\begin{proof}[Proof of Theorem \ref{thm:improved-larger}]

Apply Proposition \ref{prop:inverse-large-deg} to conclude that either $|X|$ is very small and we are done, or else there exists a polynomial $f(x)\in\Q[x]$ of degree $d\in [2,C]$ and height at most $N^C$ such that $|X\cap f([N])|\geq C^{-1}|X|$. Moreover, we have $\tau(d)\geq (1-\epsilon)\alpha^{-1}$. Hence
\[ |X|\ll |X\cap f([N])|\leq |[N]\cap f([N])|\ll N^{1/d}, \]
where the last inequality follows from a result of Walsh \cite{Wal13}, which removes the $\epsilon$ term from the exponent appearing in \cite{BP89,Hea02}. 
\end{proof}

\section{Polynomials with small value sets modulo primes}\label{sec:good-poly}

In this section we prove Theorem \ref{thm:poly-good}. First we state a result connecting the quantity $\alpha_p(f)$ to a Galois group. For a polynomial $f(x)\in\F_p[x]$ of degree $d$, denote by $R_f$  the set of roots in $\overline{\F_p(t)}$ of the polynomial $f(x)-t$. Define
\[ G_f=\Gal(\F_p(R_f)/\F_p(t)),\ \ G_f^*=\Gal(\overline{\F}_p(R_f)/\overline{\F}_p(t)). \]
In other words, $G_f$ and $G_f^*$ are the Galois groups of the splitting field of $f(x)-t$ over $\F_p(t)$ and $\overline{\F}_p(t)$, respectively. It is easy to see that $G_f^*$ is a normal subgroup of $G_f$ with $G_f/G_f^*$ cyclic. In fact, $G_f/G_f^*$ is isomorphic to $\Gal(\F_p(R_f)\cap \overline{\F}_p/\F_p)$.  For any subset $\Xi\subset G_f$, we use $\alpha(\Xi)$ to denote the proportion of elements in $\Xi$ with at least one fixed point, under the natural action on $R_f$.

\begin{lemma}[Cohen]\label{lem:cohen}
Let $f(x)\in \F_p[x]$ be a polynomial of degree $d\geq 1$. Let $\sigma G_f^*$ be the coset which is the Frobenius generator of the cyclic quotient $G_f/G_f^*$. Then 
\[ \alpha_p(f)=\alpha(\sigma G_f^*)+O_d(p^{-1/2}). \]
In particular, if $G_f=G_f^*$ then
\[ \alpha_p(f)=\alpha(G_f)+O_d(p^{-1/2}). \]
\end{lemma}

\begin{remark}
In \cite{Coh70} this is deduced from a function field version of Chebotarev density theorem. The Galois groups $G_f$ and $G_f^*$ above can be interpreted in terms of finite etale Galois coverings of $\mathbb{P}^1(\F_p)$. In this way Lemma \ref{lem:cohen} becomes a $0$-dimensional special case of Deligne's equidistribution theorem. See \cite{Kow10} for an excellent survey on this topic. This function field version of Chebotarev density theorem and related equidistribution results play an important role in proving function field analogues of certain classical analytic number theory conjectures \cite{BBR13,ABR14,Ent14}.
\end{remark}

\begin{proof}[Proof of Theorem \ref{thm:poly-good}]
Recall that the sequence of polynomials $\{f_n\}$ is defined by
\[ f_1(x)=x^2,\ \ f_{n+1}(x)=(f_n(x)+1)^2. \]
Write $G_n=G_{f_n}$, $G_n^*=G_{f_n}^*$, and $R_n=R_{f_n}$. Since any root $\alpha\in R_n$ satisfies either $f_{n-1}(\alpha)=-1+\sqrt{t}$ or $f_{n-1}(\alpha)=-1-\sqrt{t}$, we may decompose $R_n=R_n^+\cup R_n^-$ with
\[ R_n^{\pm}=\{\alpha\in R_n:f_{n-1}(\alpha)=-1\pm\sqrt{t}\}. \]
Note that both $\Gal(\F_p(R_n^+)/\F_p(\sqrt{t})$ and $\Gal(\F_p(R_n^-)/\F_p(\sqrt{t}))$ are isomorphic to $G_{n-1}$, and similarly  both $\Gal(\overline{\F}_p(R_n^+)/\overline{\F}_p(\sqrt{t})$ and $\Gal(\overline{\F}_p(R_n^-)/\overline{\F}_p(\sqrt{t}))$ are isomorphic to $G_{n-1}^*$.

Let $H_n$ and $H_n^*$ be the normal subgroup of $G_n$ and $G_n^*$ that fixes $\sqrt{t}$, so that $[G_n:H_n]=[G_n^*:H_n^*]=2$. Since $H_n$ preserves both $R_n^+$ and $R_n^-$, we get an embedding $\iota_n:H_n\hookrightarrow G_{n-1}\times G_{n-1}$ by setting the first and the second component of $\iota_n(\xi)$ to be the image of $\xi$ under the two quotient maps $H_n\rightarrow \Gal(\F_p(R_n^+)/\F_p(\sqrt{t}))$ and  $H_n\rightarrow \Gal(\F_p(R_n^-)/\F_p(\sqrt{t}))$, respectively. Similarly, we also get an embedding $\iota_n^*:H_n^*\hookrightarrow G_{n-1}^*\times G_{n-1}^*$.

We show, by induction on $n$, that when $p>2f_{n-1}(0)+2$, the embeddings $\iota_n$ and $\iota_n^*$ are in fact isomorphisms, and moreover $G_n=G_n^*$ for each $n$. The base case is clear. Now assume that $G_{n-1}=G_{n-1}^*$. To see that $\iota_n^*$ is surjective, by Lemma 15 in \cite{Fri70} it suffices to verify that for each $\lambda\in \overline{\F}_p$, at most one of the two values $-1+\sqrt{\lambda}$ and $-1-\sqrt{\lambda}$ is a branch point of $f_{n-1}$. By definition, the set of branch points of $f_{n-1}$ is
\[ \{f_{n-1}(x):x\in\overline{\F}_p,f_{n-1}'(x)=0\}. \]
This is easily computed to be the set
\[ \{f_1(0),f_2(0),\cdots,f_{n-1}(0)\}=\{0,1,4,25,\cdots\}. \]
When $p>2f_{n-1}(0)+2$, it is indeed the case that at most one of $-1+\sqrt{\lambda}$ and $-1-\sqrt{\lambda}$ can lie in this set for any $\lambda$. This shows that $H_n^*\cong G_{n-1}^*\times G_{n-1}^*\cong G_{n-1}\times G_{n-1}$. Moreover, since $H_n^*\subset H_n\subset G_{n-1}\times G_{n-1}$, we conclude that $H_n^*=H_n$ and thus $G_n^*=G_n$ as well. This completes the induction step.

With the structure of $G_n$ in hand, it is now a simple matter to write down the recursive relation
\[ \alpha_p(f_n)=\frac{1}{2}[1-(1-\alpha_p(f_{n-1}))^2]=\alpha_p(f_{n-1})-\frac{1}{2}\alpha_p(f_{n-1})^2, \]
provided that $p>2f_{n-1}(0)+2$. In fact, if $\xi\in G_n$ has a fixed point, then $\xi$ must fix $\sqrt{t}$ and thus lie in $H_n$, and moreover at least one of the two components of $\iota_n(\xi)$ has a fixed point. Finally, the bound $a_n\leq 2n^{-1}$  follows from a standard induction argument.
\end{proof}

We remark that if a polynomial $f$ has small value of $\alpha(f)$, then $f$ is necessarily highly decomposable, in the sense that $f$ should be the composition of many polynomials (each of which has degree at least $2$). On the contrary, we say that $f$ is indecomposable if it cannot be written as a composition of two polynomials of degree at least $2$. The following proposition essentially follows from results in \cite{Fri70} (similar arguments are also used in \cite{GW97}).

\begin{proposition}[Indecomposible polynomials have large value sets]\label{prop:indecomposable}
Let $f(x)\in\Z[x]$ be an indecomposable polynomial of degree $d\geq 1$. Then the average value of $\alpha_p(f)^{-1}$ as $p$ varies is at most $2$. Consequently $\alpha(f)\geq 1/2$.
\end{proposition}

\begin{proof}
Let $G$ be the Galois group of the splitting field of $f(x)-t$ over $\Q(t)$, viewed as a subgroup of the symmetric group $S_d$ on $d$ letters via its action on the $d$ roots of $f(x)-t$. Since $f$ is indecomposable, $G$ is primitive (Lemma 2 of \cite{Fri70}). Moreover, $G$ contains a $d$-cycle (Lemma 3 of \cite{Fri70}). Hence either $d$ is prime or $G$ is doubly transitive (Lemma 7 of \cite{Fri70}). In either case, the conclusion follows from Theorem \ref{thm:m-average-special}, since $\tau(d)=2$ when $d$ is prime and $(f(x)-f(y))/(x-y)\in\Q[x,y]$ is irreducible when $G$ is doubly transitive (Lemma 14 of \cite{Fri70}).
\end{proof}

\section{Further remarks}\label{sec:remarks}

\subsection{More on the sharpness of Gallagher's larger sieve}\label{sec:rem-sharp}

We point out here that Gallagher's larger sieve in its general form as stated in Theorem \ref{thm:gallagher} has the optimal bound. Indeed, if we take $A\subset [N]$ to be any subset with cardinality $Q$ and take $\mathcal{P}$ be the set of all primes between $Q$ and $N$, the general form of the larger sieve gives the sharp bound $|A|\ll Q$, because the numerator is about $N$ and the denominator is about $N/Q$. This shows that any potential improvement to Corollary \ref{cor:gallagher} must incorporate the ill distribution modulo many \textit{small} primes.

Under the assumption of Corollary \ref{cor:gallagher}, one may go over the argument in the proof of the larger sieve to find out what happens if $|A|$ is close to $N^{\alpha}$. Indeed, in the typical proof of Gallagher's larger sieve, one uses the upper and lower bounds
\begin{equation}\label{eq:larger-proof} 
\frac{|X|^2}{\alpha}\log Q\leq \sum_{x,x'\in X}\sum_{\substack{p\mid x-x'\\ p\leq Q}}\log p\leq |X|^2\log N+|X|Q,
\end{equation}
where $Q$ is about $N^{\alpha}$. 

If the upper bound is (almost) sharp, then almost all of the nonzero differences $x-x'$ should be $Q$-smooth, meaning that they do not have prime divisors larger than $Q$. For a random integer $n$, it is reasonable to expect that
\begin{equation}\label{eq:expect} 
\sum_{\substack{p\mid n\\ p\leq Q}}\log p\approx \sum_{p\leq Q}\frac{\log p}{p}\sim \log Q.
\end{equation}
If this indeed holds for almost all differences $x-x'$, then one can take $Q$ to be any small power of $N$ and deduce from \eqref{eq:larger-proof} that $|X|\ll N^{\epsilon}$.

Now consider the situation when $X$ is the set of $d$th powers up to $N$. Because of the factorization
\[ a^d-b^d=\prod_{\ell\mid d}\Phi_{\ell}(a,b), \]
where $\Phi_{\ell}$ is the cyclotomic polynomial of degree $\phi(\ell)$, we cannot expect \eqref{eq:expect} to be true for $n=a^d-b^d$. However, it is still reasonable to expect that each factor $\Phi_{\ell}(a,b)$ satisfies \eqref{eq:expect}. If so, then we obtain an upper bound in \eqref{eq:larger-proof} with $\log N$ there replaced by $\tau(d)\log Q$, which in turn implies that $\tau(d)\geq\alpha^{-1}$. This is consistent with the conclusion of Theorem \ref{thm:improved-larger}.

On the other hand, making this heuristic rigorous could be extremely hard. For example, it is an open problem to obtain a bound better than $|X|\ll N^{1/2}$ for $X\subset [N]$ with all nonzero differences $x-x'$ ($x,x'\in X$) $N^{\kappa}$-smooth, where $\kappa>0$ is sufficiently small (see \cite{EH13}).

We also point out that there are versions of Gallagher's larger sieve over arbitrary number field \cite{EEHK09,Zyw10}. One can ask similar questions about its sharpness in this general setting, and use Theorem \ref{thm:m-average-special} to formulate an improved larger sieve conjecture. We will not do so here since the case over $\Z$ is already quite interesting.

\subsection{Computing $\alpha(f)$ via Galois groups}\label{sec:rem-alpha}

The main result of this paper computes the average of $m_p(f)$ as $p$ varies, as a consequence of Chebotarev density theorem. It is natural to ask if one can compute $\alpha(f)$, the average of $\alpha_p(f)$ as $p$ varies, directly, especially since we do have such a formula for each individual $\alpha_p(f)$ as in Lemma \ref{lem:cohen}.

\begin{proposition}\label{prop:alpha-average}
Let $K$ be a number field and $f(x)\in \mathcal{O}_K[x]$ be a monic polynomial of degree $d$. Let $G=\Gal(K(R)/K(t))$, where $R$ is the set of roots of $f(x)-t$. Let $\alpha(G)$ be the proportion of elements in $G$ with at least one fixed point, under the natural action on $R$. Then
\[ \lim_{Q\rightarrow\infty}\frac{1}{\pi(Q)}\sum_{N(\pp)\leq Q}\alpha_{\pp}(f)=\alpha(G). \]
In other words, $\alpha(f)=\alpha(G)$.
\end{proposition}

\begin{remark}
Unfortunately, we are unable to use this interpretation to obtain good lower bounds on $\alpha(f)$, but see \cite{GW97} for an example where large values of $\alpha_p(f)$ are studied via Galois groups. On the other hand, we feel that any possible improvement to the bound $\alpha(f)\geq \tau(d)^{-1}$ is likely to come from studying the Galois group $G$.
\end{remark}

\begin{proof}
Write $E=K(R)$. Let $G^*=\Gal(\overline{K}(R)/\overline{K}(t))$. Let $L=E\cap \overline{K}$ be the algebraic closure of $K$ in $E$, so that $E=L(R)$ and $G^*=\Gal(E/L(t))$. By the primitive element theorem, there exists $\theta\in E$ such that $E=L(t,\theta)$. Suppose that $\theta$ satisfies the relation
\[ h_m(t)\theta^m+\cdots h_1(t)\theta+h_0(t)=0, \]
where $m=[E:L(t)]$ and $h_m(t),\cdots,h_1(t),h_0(t)$ are relatively prime polynomials over $L$. Let $h\in L[t,y]$ be the two variable polynomial defined by
\[ h(t,y)=h_m(t)y^m+\cdots h_1(t)y+h_0(t). \]
Clearly $h$ is a minimal polynomial of $\theta$, and thus $h$ is irreducible. By the definition of $L$, the polynomial $h$ is also absolutely irreducible.

Let $\pp\subset\mathcal{O}_K$ be a prime in $K$ and $\PP\subset\mathcal{O}_L$ be a prime in $L$ lying above $\pp$. By Lemmas \ref{lem:noether} and \ref{lem:abs-irred}, $h_{\PP}\in\kappa_{\PP}[t,y]$ remains absolutely irreducible for all but finitely many $\PP$. Let $\theta_{\PP}\in\overline{\kappa_{\PP}(t)}$ be an element satisfying $h_{\PP}(t,\theta_{\PP})=0$, so that $E_{\PP}=\kappa_{\PP}(t,\theta_{\PP})$ is a degree $m$ field extension of $\kappa_{\PP}(t)$ with $E_{\PP}\cap \overline{\kappa_{\pp}}=\kappa_{\PP}$. Since $E/L(t)$ is Galois, all roots of $h(t,y)$ in $\overline{L(t)}$ lie in $E$. This implies that all roots of $h_{\PP}(t,y)$ in $\overline{\kappa_{\PP}(t)}$ lie in $E_{\PP}$ for all but finitely many $\PP$, and thus $E_{\PP}/\kappa_{\PP}(t)$ is also Galois. Note that there is a natural isomorphism $G^*=\Gal(E/L(t))\cong \Gal(E_{\PP}/\kappa_{\PP}(t))$, since an element in either Galois group is determined by its image of $\theta$ or $\theta_{\PP}$.

Now we look at the polynomial $f(x)-t$. Since it factors into linear factors over $E$, its reduction $f_{\pp}(x)-t$ factors into linear factors over $E_{\PP}$ for all but finitely many $\PP$. By an abuse of notation, we will continue to write $R$ for the set of roots of $f_{\pp}(x)-t$ in $\overline{\kappa_{\pp}(t)}$. Therefore the splitting fields $\kappa_{\PP}(R)$ and $\kappa_{\pp}(R)$ are contained in $E_{\PP}$. On the other hand, since $\theta\in K(R)$ and $L\subset K(R)$, we have $\theta_{\PP}\in\kappa_{\pp}(R)$ and $\kappa_{\PP}\subset\kappa_{\pp}(R)$ for all but finitely many $\PP$. This shows that $\kappa_{\pp}(R)=E_{\PP}$.

Let $\sigma_{\PP}G^*$ be the coset which is the inverse image of the Frobenius automorphism $\sigma_{\PP}$ under the quotient map 
\[ \Gal(E_{\PP}/\kappa_{\pp}(t))\twoheadrightarrow \Gal(\kappa_{\PP}(t)/\kappa_{\pp}(t))=\Gal(\kappa_{\PP}/\kappa_{\pp}), \]
which has kernel $\Gal(E_{\PP}/\kappa_{\PP}(t))=G^*$. By Lemma \ref{lem:cohen}, we have
\[ \alpha_{\pp}(f)=\alpha(\sigma_{\PP}G^*)+O_d(N(\pp)^{-1/2}). \]
Note that the quantity $\alpha(\sigma_{\PP}G^*)$ does not depend on the choice $\PP$. Via the inclusion $\Gal(\kappa_{\PP}/\kappa_{\pp})\hookrightarrow \Gal(L/K)$, we may view $\sigma_{\PP}$ as an element in $\Gal(L/K)$ and $\sigma_{\PP}G^*$ as a coset in $G$. By Chebotarev density theorem, the cosets $\sigma_{\PP}G^*$ become equidistributed in $G$ as $\pp$ varies. Therefore $\alpha(f)=\alpha(G)$ as desired.
\end{proof}

For a generic polynomial of degree $d$, the Galois group $G$ in Proposition \ref{prop:alpha-average} is the full symmetric group $S_d$, and thus 
\[ \alpha(f)=1-\frac{1}{2}+\frac{1}{6}-\frac{1}{24}+\cdots+\frac{(-1)^{d-1}}{d!}  \]
for a typical $f$ of degree $d$. Moreover, this quantity tends to $1-e^{-1}$ as $d\rightarrow\infty$.

For $d\leq 4$, we have the following sharp lower bounds.

\begin{proposition}[Polynomials of small degree]\label{prop:small-degree}
For a positive integer $d$, let $\alpha_d$ be the smallest possible value of $\alpha(f)$, where $f\in\Q[x]$ is a polynomial of degree $d$. Then $\alpha_2=1/2$, $\alpha_3=2/3$, and $\alpha_4=3/8$.
\end{proposition}

\begin{proof}
For $d=2$ this is obvious. Suppose that $d\in\{3,4\}$. Let $G$ be the Galois group as in Proposition \ref{prop:alpha-average}. We claim that $G\neq \Z/d\Z$, the cyclic group of order $d$. In fact, for $t\in\Z$ sufficiently large, the polynomial $f(x)=t$ has at least one real root and at least one non-real root. Let $\alpha\in\R$ be a real root of $f(x)=t$. Then the splitting field of $f(x)-t$ contains properly the subfield $\Q(\alpha)$, and thus has degree larger than $d$ over $\Q$. This shows that the Galois group of $f(x)-t$ is not $\Z/d\Z$ for all $t$ sufficiently large. The fact that $G\neq\Z/d\Z$ then follows from Hilbert's irreducibility theorem. Now that $G\subset S_d$ is transitive and $G\neq \Z/d\Z$, the only possibilities are $G=S_3$ when $d=3$ and $G\in\{S_4,A_4,D_4\}$ when $d=4$. The conclusion follows by computing $\alpha(G)$ for these choices of $G$.  
\end{proof}

Not surprisingly, the nature of $\alpha_d$ depends not only on the size of $d$, but also the arithmetic of $d$ (cf. Proposition \ref{prop:indecomposable}). In general, given a transitive subgroup $G\subset S_d$, we do not know how to tell whether $G$ can be realized as a Galois group as in Proposition \ref{prop:alpha-average}.

\bibliographystyle{plain}
\bibliography{sieve}{}

\end{document}